\documentclass[leqno, a4paper, 12pt]{article}

\usepackage{amssymb, amscd, amsthm}

\theoremstyle{plain}
\newtheorem{theorem}{Theorem}[section]
\newtheorem{lemma}[theorem]{Lemma}
\newtheorem{proposition}[theorem]{Proposition}
\newtheorem{corollary}[theorem]{Corollary}

\theoremstyle{definition}

\newtheorem{remark}[theorem]{Remark}

\newtheorem{examples}[theorem]{Examples}

\newcommand\bC{{\mathbb C}}

\newcommand\bG{{\mathbb G}}

\newcommand\bL{{\mathbb L}}

\newcommand\bP{{\mathbb P}}

\newcommand\bR{{\mathbb R}}

\newcommand\bZ{{\mathbb Z}}

\newcommand\bPic{{\bf Pic}}

\newcommand\cA{{\mathcal A}}

\newcommand\cL{{\mathcal L}}
\newcommand\cM{{\mathcal M}}

\newcommand\cO{{\mathcal O}}

\newcommand\cQ{{\mathcal Q}}

\newcommand\fm{\mathfrak{m}}

\newcommand\wA{\widehat{A}}
\newcommand\wT{\widehat{T}}

\newcommand\charc{{\rm char}}

\newcommand\et{\rm{\acute{e}t}}
\newcommand\id{{\rm id}}

\newcommand\red{{\rm red}}

\newcommand\sep{{\rm sep}}

\newcommand\R{{\rm R}}

\newcommand\Ext{{\rm Ext}}

\newcommand\Hom{{\rm Hom}}

\newcommand\Kern{{\rm Ker}}

\newcommand\Pic{{\rm Pic}}

\newcommand\Proj{{\rm Proj}}

\newcommand\Spec{{\rm Spec}}

\newcommand\Sym{{\rm Sym}}

\title{Which algebraic groups are Picard varieties?}

\author{Michel Brion}

\date{}

\begin{document}

\maketitle
 
\begin{abstract}
We show that every connected commutative algebraic group 
over an algebraically closed field of characteristic $0$ 
is the Picard variety of some projective variety having
only finitely many non-normal points. In contrast, no Witt 
group of dimension at least $3$ over a perfect field of 
prime characteristic is isogenous to a Picard variety 
obtained by this construction.
\end{abstract}

\section{Introduction and statement of the main results}
\label{sec:int}

With any proper scheme $X$ over a field $k$, one associates the 
Picard scheme $\bPic_{X/k}$ and its neutral component $\bPic^0_{X/k}$, 
a connected group scheme of finite type which parameterizes the
algebraically trivial invertible sheaves on $X$. When $k$ is perfect,
the reduced neutral component of $\bPic_{X/k}$ is an algebraic goup, 
classically known as the Picard variety $\Pic^0(X)$. One may ask 
whether any connected commutative algebraic group can be obtained 
in this way. In this article, we obtain a positive answer to that 
question when $k$ is algebraically closed of characteristic $0$, 
and a negative partial answer in prime characteristics. The analogous 
question for the reduced neutral component of the automorphism group 
scheme is answered in the positive by \cite[Thm.~1]{Brion}.

By general structure results, every connected commutative algebraic
group $G$ over a perfect field sits in a unique exact sequence
$0 \to U \times T \to G \to A \to 0$,
where $U$ is a connected unipotent algebraic group, $T$ a torus, 
and $A$ an abelian variety. Conversely, given such an exact sequence, 
we shall construct a projective variety $X$ such that 
$\Pic^0(X) \cong G$, under additional assumptions on the affine part
$U \times T$. Our result holds more generally in the setting of 
the relative Picard functor (see \cite{BLR,Kleiman}):

\begin{theorem}\label{thm:first}
Let $S$ be a locally noetherian scheme, and 
\begin{equation}\label{eqn:first} 
0 \longrightarrow V \times T \longrightarrow G 
\longrightarrow A \longrightarrow 0 
\end{equation}
an exact sequence of commutative $S$-group schemes, where $V$ 
is a vector group, $T$ a quasi-split torus, and $A$ an abelian scheme. 
Then there exists a proper flat $S$-scheme $X$ with integral 
geometric fibers, such that $G \cong \bPic^0_{X/S}$. Moreover,
$X$ may be taken locally projective over $S$, if $A$ is locally
projective.
\end{theorem}

Here a vector group is the additive group scheme of a locally free
sheaf of finite rank; a quasi-split torus is a group scheme
$T$ such that the pull-back $T_{S'}$ under some finite \'etale Galois 
cover $S' \to S$ with group $\Gamma$ is isomorphic to a direct product 
of finitely many copies of $\bG_{m,S'}$ which are permuted by $\Gamma$. 

Under the assumptions of that theorem, we now sketch 
how to construct the desired scheme $X$ from the exact sequence 
(\ref{eqn:first}). We use the process of pinching studied in 
\cite{Ferrand}; more specifically, we obtain $X$ by pinching 
an appropriate smooth $S$-scheme $X'$ along a finite subscheme $Y'$ 
via a morphism $\psi : Y' \to Y$. 
We then have an exact sequence
\[ 0 \longrightarrow \bG_{m,S} \longrightarrow V_{Y'}^* 
\longrightarrow (\Pic_{X'/S},Y') \longrightarrow
\Pic_{X'/S} \longrightarrow 0, \]
where $V_{Y'}^*$ is a smooth affine group scheme with connected 
fibers, defined by $V_{Y'}^*(S') = \cO(Y'_{S'})^*$ for any scheme $S'$ 
over $S$; $\Pic_{X'/S}$ stands for the relative Picard functor, 
and $(\Pic_{X'/S},Y')$ parameterizes the
invertible sheaves on $X'$, rigidified along $Y'$ (see 
\cite[8.1]{BLR}). There is of course an analogous sequence for 
$(X,Y)$; in addition, one easily obtains an isomorphism of
rigidified Picard functors $(\Pic_{X/S},Y) \cong (\Pic_{X'/S},Y')$. 
All of this yields an exact sequence 
\begin{equation}\label{eqn:funct}
0 \longrightarrow V_{Y'}^*/\psi^*(V_Y^*) \longrightarrow \Pic_{X/S} 
\longrightarrow \Pic_{X'/S} \longrightarrow 0.
\end{equation}
It remains to find $X'$, $Y'$ and $\psi$ so that (\ref{eqn:funct}) 
gives back the exact sequence (\ref{eqn:first}). For this, 
we use a result of \"Onsiper: every extension of an abelian scheme 
by the direct product of a vector group and a split torus can be 
constructed as a rigidified Picard functor (see \cite{Onsiper}). 
A slight modification of that construction yields the desired objects; 
note that \cite{Onsiper} uses the notion of rigidifier as in 
\cite{Raynaud}, which is weaker than that of \cite{BLR}.

Over an algebraically closed field of characteristic $0$, every 
connected commutative unipotent group is a vector group, and 
every torus is (quasi-)split; hence any connected commutative 
algebraic group is the Picard variety of some projective variety
with finite singular locus. But this does not extend to prime 
characteristics:

\begin{theorem}\label{thm:second}
Let $W_n$ denote the Witt group of dimension $n$ over a perfect 
field $k$ of characteristic $p > 0$. Then $W_n$ is not isogenous 
to the Picard variety of any projective variety with finite
non-normal locus, if $p \geq 5$ and $n \geq 2$ (resp. 
$p \leq 3$ and $n \geq 3$).
\end{theorem}

It should be noted that the affine part of the Picard variety 
of any proper reduced scheme $X$ over a perfect field $k$ has been 
described by Geisser in \cite{Geisser}. In particular, the 
maximal torus of $\Pic^0(X)$ has cocharacter module isomorphic 
to $H^1_{\et}(X_{\bar{k}},\bZ)$ as a Galois module
(see \cite[Thm.~1]{Geisser}, and \cite[Thm.~4.1.7]{Alexeev}
for a closely related result). We do not know whether all tori
(or equivalently, all Galois modules) can be obtained in this way. 
When the non-normal locus of $X$ is finite, the maximal torus of 
$\Pic^0(X)$ must be stably rational, see Remark \ref{rem:stab}.

This article is organized as follows. In Section \ref{sec:ppf},
we begin by gathering results taken from \cite{Ferrand} about 
pinching and Picard groups; then we obtain the exact sequence 
(\ref{eqn:funct}) together with representability of the 
associated Picard functors under suitable assumptions. Section 
\ref{sec:ext} constructs some extensions of abelian schemes by 
adapting the results of \cite{Onsiper}; it concludes with the proof
of Theorem \ref{thm:first}. In Section \ref{sec:rel}, 
we study the quotients $\mu^B/\mu^A$, where $A \subset B$ are 
artinian algebras over a field and $\mu^A \subset \mu^B$ denote 
the associated unit group schemes. These quotients are exactly 
the affine parts of Picard varieties of projective varieties 
with finite non-normal locus, see Proposition \ref{prop:pic}. 
We conclude with the proof of Theorem \ref{thm:second}.

\section{Pinching and Picard functor}
\label{sec:ppf}

\subsection{Pinched schemes}
\label{subsec:fer}

Throughout this section, we fix a locally noetherian base scheme $S$. 
Schemes are assumed to be separated and of finite type over $S$ 
unless otherwise mentioned. 

Let $X'$ be a scheme, $\iota': Y' \to X'$  the inclusion of 
a closed subscheme, and $\psi : Y' \to Y$ a finite morphism. 
We assume that the natural map $\cO_Y \to \psi_*(\cO_{Y'})$ is injective; 
in particular, $\psi$ is surjective. We also assume that $X'$ and $Y$ 
satisfy the following condition:

\smallskip

\noindent
(AF) \quad Every finite set of points is contained in an open 
affine subscheme.

\smallskip

\noindent
Under these assumptions, there exists a cocartesian diagram of schemes
\begin{equation}\label{eqn:pin}
\CD
Y' @>{\iota'}>> X' \\
@V{\psi}VV @V{\varphi}VV \\
Y @>{\iota}>> X,\\
\endCD
\end{equation}
where $\iota$ is a closed immersion, $\varphi$ is finite, 
and $X$ satisfies (AF). Moreover, $\varphi$ 
induces an isomorphism $X' \setminus Y' \to X \setminus Y$; in 
particular, $\varphi$ is surjective. We say that 
\emph{$X$ is obtained by pinching $X'$ along $Y'$ via $\psi$}.

These results follow from \cite[Thm.~5.4, Prop.~5.6]{Ferrand}, except 
for the assertion that $X$ is of finite type over $S$, which is a 
consequence of \cite[Chap.~V, \S 1, no. 9, Lem.~5]{Bourbaki}.
If in addition $X'$ is proper over $S$, then so is $X$ (since 
$\varphi : X' \to X$ is finite and surjective). But projectivity is
not preserved under pinching, as shown by the examples in
\cite[Sec.~6]{Ferrand}.

Since the formation of $X$ is Zariski local on $S$, we may replace 
(AF) with a slightly weaker condition:

\smallskip

\noindent
(LAF) \quad $S$ is covered by open subschemes $S_i$ such that every finite 
set of points over $S_i$ is contained in an open affine subscheme.

\smallskip

This condition holds in particular for locally projective $S$-schemes.

\subsection{Their invertible sheaves}
\label{subsec:inv} 

With the notation and assumptions of Subsection \ref{subsec:fer}, 
the data of an invertible sheaf $\cL$ on $X$ is equivalent to that of 
a triple $(\cL',s',\cM)$, where $\cL'$ (resp. $\cM$) is an invertible 
sheaf on $X'$ (resp. $Y$), and $s' : \psi^*(\cM) \to \iota'^*(\cL')$ 
is an isomorphism.
Namely, one associates with $\cL$ the sheaves 
$\cL' := \varphi^*(\cL)$, $\cM := \iota^*(\cL)$ and the isomorphism
\[ \psi^*(\cM) = \psi^* \iota^*(\cL) \longrightarrow 
\iota'^* \varphi^*(\cL) = \iota'^*(\cL') \]
arising from the commutative diagram (\ref{eqn:pin}).

Moreover, the isomorphisms $\cL_1 \to \cL_2$ are equivalent to the pairs
$(u,v)$, where $u : \cL'_1 \to \cL'_2$, $v : \cM_1 \to \cM_2$
are isomorphisms such that the diagram
\[ \CD
\psi^*(\cM_1) @>{s'_1}>> \iota'^*(\cL'_1) \\
@V{\psi^*(v)}VV @V{\iota'^*(u)}VV \\
\psi^*(\cM_2) @>{s'_2}>> \iota'^*(\cL'_2) \\
\endCD \]
commutes, with an obvious notation.

These results are consequences of \cite[Thm.~2.2]{Ferrand} 
when $X'$ is affine; the general case follows by using the fact 
that $\varphi$ is affine, as alluded to in [loc.~cit., 7.4] 
and explained in detail in \cite[Thm.~3.13]{Howe}.

In particular, for any $s' \in \cO(Y')^*$ (the unit group of 
the ring of global sections $\cO(Y')$), the triple 
$(\cO_{X'}, s', \cO_Y)$ corresponds to an invertible sheaf on $X$, 
which is trivial if and only if $s' = \iota'^*(u) \psi^*(v)$ 
for some $u \in \cO(X')^*$ and $v \in \cO(Y)^*$.

Also, an invertible sheaf $\cL'$ over $X'$ is the pull-back of
some invertible sheaf on $X$ if and only if 
$\iota'^*(\cL') \cong \psi^*(\cM)$ for some invertible sheaf 
$\cM$ on $Y$.

\subsection{Their Picard functor}
\label{subsec:pic}

We keep the notation and assumptions of Subsection \ref{subsec:fer}, 
and assume in addition the following two conditions:

\smallskip

\noindent
(PF) \quad The structure map $f' : X' \to S$ is proper and flat 
with integral geometric fibers.

\smallskip

\noindent
(FF) \quad The structure maps $g: Y \to S$ and $g' : Y' \to S$ 
are finite and faithfully flat.

\smallskip

The latter condition implies that $Y$ satisfies (LAF). Also,
by \cite[Prop.~7.8.6]{EGAIII}, the condition (PF) 
yields that $f'_*(\cO_{X'}) = \cO_S$ universally. 

We now recall some notions and results from \cite[\S 8.1]{BLR}.
We denote by $\Pic_{X'/S}$ the relative Picard functor, i.e., 
the fppf sheaf associated with the functor $S' \mapsto \Pic(X'_{S'})$, 
where $X'_{S'} := X' \times_S S'$. Since 
$f'^*: \cO(S') \to \cO(X'_{S'})$ is an isomorphism for any $S$-scheme
$S'$, the natural map $\cO(X'_{S'}) \to \cO(Y'_{S'})$ is injective,
and hence $Y'$ is a rigidifier of $\Pic_{X'/S}$. Also, the
functor $S' \mapsto \cO(Y'_{S'})$ is represented by a locally free
ring scheme $V_{Y'}$, and the subfunctor of units,
$S' \mapsto \cO(Y'_{S'})^*$, by a group scheme, open in $V_{Y'}$.
Clearly, $V_{X'} \cong \bG_{a,S}$ and $V_{X'}^* \cong \bG_{m,S}$.
Also, note that
\[ V_{Y'}^* = \R_{Y'/S}(\bG_{m,Y'}), \]
where $\R$ denotes the Weil restriction. 
We have an exact sequence of sheaves for the \'etale topology
\begin{equation}\label{eqn:exa}
0  \longrightarrow V_{X'}^* \longrightarrow V_{Y'}^* \longrightarrow
(\Pic_{X'/S}, Y') \longrightarrow \Pic_{X'/S} \longrightarrow 0,
\end{equation}
where $(\Pic_{X'/S},Y')$ denotes the sheaf of isomorphism classes of 
invertible sheaves on $X'$, ridigified along $Y'$. 

We record some easy additional properties of the unit group 
scheme $V_{Y'}$:

\begin{lemma}\label{lem:units}
With the above notation, we have:

\smallskip

\noindent
{\rm (i)} $V_{Y'}^*$ is a smooth affine group scheme with connected 
fibers. 

\smallskip

\noindent
{\rm (ii)} If $Y'$ is the disjoint union of two closed
subschemes $Y'_1, Y'_2$, then $V_{Y'} \cong V_{Y'_1} \times_S V_{Y'_2}$.
\smallskip

\noindent
{\rm (iii)} $V_{Y'}^*$ is a torus if and only if $Y'$ is \'etale over $S$.

\end{lemma}

\begin{proof}
(i) Since $V_{Y'}$ is smooth, so is its open subscheme $V_{Y'}^*$.
Also, Weil restriction preserves affineness in view of (FF) and 
\cite[Chap.~I, \S 1, Prop.~6.6]{Demazure-Gabriel}; in particular,
$V_{Y'}^*$ is affine. Its fibers are connected by
\cite[Prop.~2.4.3]{Raynaud}.

(ii) is readily checked.

(iii) If $V_{Y'}^*$ is a torus, then so are its fibers
$(V_{Y'}^*)_s = \cO(Y'_s)^*$. It follows readily that
the $k(s)$-algebra $\cO(Y'_s)$ is separable
(see Proposition \ref{prop:car} below for a more general result). 
Hence $Y'$ is \'etale over $S$. To show the converse implication, 
we may replace $S$ (resp. $Y'$) with $Y'$ 
(resp. $Y'$ with $Y' \times_S Y'$). Then $g': Y' \to S$ has 
a section, so that $Y' = S \sqcup Y''$ for some scheme $Y''$, 
finite and \'etale over $S$. Thus, 
$V_{Y'}^* \cong \bG_{m,S} \times V_{Y''}^*$ and we conclude 
by induction.

\end{proof}

Next, observe that the structure map $f : X \to S$ also satisfies 
(PF): the properness has already been observed, while the flatness 
and the assertion on geometric fibers follow from 
\cite[Thm.~3.11]{Howe}. Thus, $Y$ is a rigidifier of $\Pic_{X/S}$ 
and the latter sits in an exact sequence of sheaves 
for the \'etale topology, analogous to (\ref{eqn:exa}). Moreover, 
both sequences sit in a commutative diagram
\begin{equation}\label{eqn:com}
\CD
0 @>>> \bG_{m,S} @>>> V_Y^* @>>> (\Pic_{X/S},Y) @>>> \Pic_{X/S} @>>> 0 \\
& & @V{\id}VV @V{\psi^*}VV @V{\varphi^*}VV @V{\varphi^*}VV \\
0 @>>> \bG_{m,S} @>>> V_{Y'}^* @>>> (\Pic_{X'/S},Y') @>>> \Pic_{X'/S} @>>> 0. \\
\endCD 
\end{equation}

We may now state a key observation:

\begin{lemma}\label{lem:rig}
The map $\varphi^* : (\Pic_{X/S},Y) \to (\Pic_{X'/S},Y')$ is an isomorphism.
\end{lemma}

\begin{proof}
Consider an arbitrary $S$-scheme $S'$. Then the square obtained from
(\ref{eqn:pin}) by base change to $S'$ is still cocartesian in view of
\cite[Thm.~3.11]{Howe}. Thus, the invertible sheaves on 
$X_{S'}$ can be described as in Subsection \ref{subsec:inv}, 
in view of \cite[Thm.~3.13]{Howe}. So it suffices to show that
\[ \varphi^* : (\Pic(X),Y) \to (\Pic(X'),Y') \]
is an isomorphism. Here $(\Pic(X),Y)$ denotes the group of isomorphism
classes of pairs $(\cL,\alpha)$, where $\cL$ is an invertible sheaf
on $X$, and $\alpha : \cO_Y \to \iota^*(\cL)$ is an isomorphism.

Let $(\cL',\alpha')$ be an invertible sheaf on $X'$, rigidified
along $Y'$. Then the triple $(\cL',\alpha',\cO_Y)$ corresponds by
Subsection \ref{subsec:inv} to an invertible sheaf $\cL$ on $X$ 
such that $\varphi^*(\cL) = \cL'$ and $\iota^*(\cL) = \cO_Y$. 
Moreover, $\varphi^*(\cL, 1) \cong (\cL',\alpha')$. Thus, 
$\varphi^*$ is surjective.

Next, let $(\cL,\alpha)$ be an invertible sheaf on $X$ rigidified
along $Y$, such that $\varphi^*(\cL,\alpha)$ is trivial in
$(\Pic(X'),Y')$. In particular, $\varphi^*(\cL) \cong \cO_{X'}$ and
$\iota^*(\cL) \cong \cO_Y$. Thus, $\cL$ is isomorphic to the
invertible sheaf associated with a triple
$(\cO_{X'}, s',\cO_Y)$, where $s' \in \cO(Y')^*$. Then 
$\alpha \in \cO(Y)^*$; moreover, replacing $(\cO_{X'},s',\cO_Y)$
with the isomorphic triple $(\cO_{X'}, s' \psi^*(v), \cO_Y)$ for 
$v \in \cO(Y)^*$ replaces $\alpha$ with $\alpha v$.  
Thus, $(\cL,\alpha)$ is isomorphic to $(\cO_X,1)$, and $\varphi^*$
is injective.
\end{proof}

Lemma \ref{lem:rig} and the commutative diagram (\ref{eqn:com}) 
yield readily the following:

\begin{corollary}\label{cor:long}
We have an exact sequence of sheaves for the \'etale topology
\[ \CD
0 @>>> V_Y^* @>{\psi^*}>> V_{Y'}^* @>>> \Pic_{X/S}
@>{\varphi^*}>> \Pic_{X'/S} @>>> 0.
\endCD \]
\end{corollary}

\subsection{Their Picard scheme}
\label{subsec:rep}

We keep the assumptions (PF) and (FF) of Subsection \ref{subsec:pic},
and assume in addition that $X'$ is locally projective over $S$.

\begin{proposition}\label{prop:rep}

{\rm (i)} $X$ is locally projective over $S$ as well.

\smallskip

\noindent
{\rm (ii)} The Picard functors $\Pic_{X'/S}$, $\Pic_{X/S}$ are 
represented by group schemes $\bPic_{X'/S}$, $\bPic_{X/S}$ 
which are locally of finite type.

\smallskip

\noindent
{\rm (iii)} Assume in addition the following condition:

\smallskip

\noindent
{\rm (R)} \quad The homomorphism of group schemes 
$\psi^*: V_Y^* \to V_{Y'}^*$ is a closed immersion 
and its cokernel is represented by a group scheme.

\smallskip

\noindent
Then the latter group scheme sits in an exact sequence
\begin{equation}\label{eqn:rep}
\CD
0 @>>> V_{Y'}^*/\psi^*(V_Y^*) @>>> \bPic_{X/S} @>{\varphi^*}>>
\bPic_{X'/S} @>>> 0.
\endCD
\end{equation}
\end{proposition}

\begin{proof}
(i) We may assume that $X'$ has an $S$-ample invertible sheaf $\cL'$.
In view of (FF), $\iota'^*(\cL')$ is trivial on the pull-back of 
some open affine covering of $S$. Thus, we may further assume that 
$\iota'^*(\cL') \cong \cO_{Y'}$; then by Subsection \ref{subsec:inv}, 
$\cL' \cong \varphi^*(\cL)$ for some invertible sheaf $\cL$ on $X$.
Since $\varphi$ is finite, $\cL$ is $S$-ample.

(ii) The assertion on $\Pic_{X'/S}$ is a consequence of (PF) and 
the local projectivity assumption in view of \cite[8.2 Thm.~1]{BLR} 
(see also \cite[Thm.~9.4.8]{Kleiman}). The assertion on $\Pic_{X/S}$ 
follows similarly in view of (i).

(iii) is a direct consequence of Corollary \ref{cor:long}.
\end{proof}

\begin{remark}\label{rem:rep}
The assumption (R) is satisfied when $S = \Spec(k)$ for a field
$k$, see the next remark. This assumption also holds when $\psi$ 
admits a section $\sigma$ (in view of the exact sequence 
$V_Y^* \stackrel{\psi^*}{\longrightarrow} V_{Y'}^*
\stackrel{\sigma^*}{\longrightarrow} V_Y^*$) or when $Y$ is
\'etale over $S$ (then $V_Y^*$ is a torus and the assertion
follows from \cite[Exp.~IX, Cor.~2.5]{SGA3}). 
\end{remark}

\begin{remark}\label{rem:field}
Consider the case where $S= \Spec(k)$, where $k$ is a field. 
Then the assumptions (PF), (FF) and of local projectivity 
just mean that $X'$ is a projective $k$-variety equipped 
with a finite subscheme $Y'$ and with a morphism 
$\psi : Y' \to Y$ such that 
$\cO_Y \hookrightarrow \psi_*(\cO_{Y'})$.
(By a variety, we mean a geometrically integral scheme.) 
Moreover, the group scheme $V_Y^*$ represents the functor
$R \mapsto (R \otimes_k A)^*$ from $k$-algebras to groups, 
where $A := \cO(Y)$ is an artinian $k$-algebra.
We shall rather denote $V_Y^*$ by $\mu^A$, as in 
\cite[Chap.~II, \S 1, 2.3]{Demazure-Gabriel}; then
$\mu^A$ is a connected affine algebraic group with Lie algebra
the vector space $A$ equipped with the trivial bracket. This
group is also considered in \cite{Russell}, where it is denoted 
by $\bL_A$.

Let $A' := \cO(Y')$; then the injective homomorphism of algebras 
$A \to A'$ induces a homomorphism of algebraic groups
$\psi^*: \mu^A \to \mu^{A'}$ which is a closed immersion 
(since $\psi^*$ is injective on points over the algebraic closure 
of $k$, and on Lie algebras). Thus, the cokernel of $\psi^*$
is represented by a connected affine algebraic group, that we
denote by $\mu^{A'/A}$. So the condition (R) is satisfied, and
(\ref{eqn:rep}) yields an exact sequence
\[ \CD
0 @>>>  \mu^{A'/A} @>>> \bPic_{X/k} @>{\varphi^*}>>
\bPic_{X'/k} @>>> 0.
\endCD \]
The analogous sequence for Picard groups is well-known (see e.g. 
\cite[Prop.~21.8.5]{EGAIV}).

If $X'$ is geometrically normal, then $\bPic^0_{X'/k}$ is projective
by \cite[Thm.~9.5.4]{Kleiman}. Thus, $\mu^{A'/A}$ is the affine
part of the Picard variety of $X$, if in addition $k$ is perfect.

Finally, all the above results extend without change to the case 
where $X'$ is a proper variety satisfying (AF). Indeed, 
the Picard functor $\Pic_{X'/k}$ is still represented by a scheme 
locally of finite type, in view of \cite[Cor.~9.4.18.3]{Kleiman}.
\end{remark}

Returning to the notation and assumptions of Proposition 
\ref{prop:rep}, assume that $Y$ is the disjoint union of two closed
subschemes $Y_1, Y_2$. Then we also have $Y' = Y'_1 \sqcup Y'_2$,
where $Y'_i := \psi^{-1}(Y_i)$ for $i = 1,2$. We may pinch $X'$ along 
the restriction $\psi_1: Y'_1 \to Y_1$ to obtain a scheme $X_1$
satisfying all the assumptions of Subsection \ref{subsec:pic}.
Moreover, the induced morphism $Y'_2 \to X_1$ is a closed immersion
(since $Y'_2 \subset X' \setminus  Y'_1$), and $X$ is obtained
by pinching $X_1$ along the restriction $\psi_2: Y'_2 \to Y_2$.
Likewise, $X$ is obtained by pinching $X_2$ along 
$\psi_1: Y'_1 \to Y_1$; this yields a commutative diagram
\[ \CD
X' @>{\varphi'_1}>> X_1 \\
@V{\varphi'_2}VV @V{\varphi_1}VV \\
X_2 @>{\varphi_2}>> X. \\
\endCD \]

\begin{lemma}\label{lem:fib}
With the above notation and assumptions, the map
\[ \varphi_1 \times \varphi_2: \bPic_{X/S} \longrightarrow
\bPic_{X_1/S}\times_{\bPic_{X'/S}} \bPic_{X_2/S} \]
is an isomorphism.
\end{lemma}

This result follows easily from the exact sequence (\ref{eqn:rep}) 
and from the analogous exact sequences for $\varphi_1$ and $\varphi_2$.
It will be used in the proof of Theorem \ref{thm:first}, to reduce
to the case where $U$ or $T$ is trivial.

\section{Some extensions of abelian schemes}
\label{sec:ext}

\subsection{Extensions by vector groups}
\label{subsec:vect}

Throughout this section, we keep the standing assumptions of
Section \ref{sec:ppf} on schemes. All group schemes are assumed 
to be commutative.

Let $A$ be an abelian scheme over $S$. By
\cite[Thm.~1.9]{Faltings-Chai}, $A$ has a dual abelian scheme $\wA$,
and both satisfy (LAF). Also, recall that $\wA$ is locally projective
if so is $A$ (see e.g. \cite[Rem.~9.5.24]{Kleiman}). 

Consider a locally free sheaf $\cQ$ of finite rank 
over $S$ and denote by $V = V(\cQ)$ its total space, i.e., the affine 
$S$-scheme associated with the sheaf of $\cO_S$-algebras 
$\Sym_{\cO_S}(\cQ)$. Then $V$ is a vector group over $S$, i.e., 
a group scheme locally isomorphic to a direct product of copies 
of $\bG_{a,S}$ and equipped with an action of $\bG_{m,S}$ which 
restricts to the multiplication on each $\bG_{a,S}$. For example,
if $Y$ is a finite faithfully flat $S$-scheme, then 
$V_Y = V(g_*(\cO_Y))$ with the notation of Subsection 
\ref{subsec:pic}.

By \cite[Chap.~I, (1.9)]{Mazur-Messing}, any extension of 
$S$-group schemes 
\begin{equation}\label{eqn:vect}
0 \longrightarrow V \longrightarrow G \longrightarrow A 
\longrightarrow 0
\end{equation}
is classified by a morphism of $S$-group schemes
\[ \gamma : V(\omega_{\wA}^{\vee}) \longrightarrow V, \]
where $\omega_{\wA}$ denotes the sheaf of (relative) differential 
$1$-forms on $\wA$, and $\omega_{\wA}^{\vee}$ its dual.
(Note that the convention of \cite{Mazur-Messing} for vector groups
is dual to ours). When we take into account the structure of
vector group of $V$ (or equivalently, the $\bG_{m,S}$-action on 
that group scheme), the morphism $\gamma$ is in addition
$\bG_{m,S}$-equivariant, i.e., it comes from a morphism of 
locally free sheaves $\cQ \to \omega_{\wA}^{\vee}$. For simplicity,
we still denote the dual morphism by 
\[ \gamma : \omega_{\wA} \longrightarrow \cQ^{\vee}. \] 
Let $I_S(\cQ^{\vee})$ denote the affine $S$-scheme 
associated with the sheaf of $\cO_S$-algebras 
$\cO_S \oplus \varepsilon \cQ^{\vee}$, where $\varepsilon^2 = 0$,
and define similary $I_S(\omega_{\wA})$. 
Then the above morphism $\gamma$ yields a morphism of schemes
\[ I_S(\gamma) : I_S(\cQ^{\vee}) \longrightarrow I_S(\omega_{\wA}). \]
Also, $I_S(\omega_{\wA})$ may be viewed as a closed subscheme of
$\wA$, namely, the first infinitesimal neighborhood of the zero
section. Thus, $I_S(\gamma)$ may be identified with a morphism
$I_S(\cQ^{\vee}) \to \wA$ with image supported in the zero section.
We also have a closed immersion 
$I_S(\cQ^{\vee}) \to V(\cQ^{\vee})$ 
with image the first infinitesimal neighborhood of the zero section. 
Viewing $V(\cQ^{\vee})$ as an open subscheme of the projective space 
$\bP(\cQ^{\vee} \oplus \cO_S) := 
\Proj \, \Sym_{\cO_S}(\cQ^{\vee} \oplus \cO_S)$,
we obtain a closed immersion
\[ \iota' : Y'  := I_S(\cQ^{\vee}) \longrightarrow 
\wA \times_S \bP(\cQ^{\vee} \oplus \cO_S) =: X'. \]
Let $\psi: Y' \to S =: Y$ denote the structure map. Then
all the assumptions of Subsection \ref{subsec:pic}
are satisfied, and hence we may form the pinching diagram
(\ref{eqn:pin}). Moreover, the Picard functors $\Pic_{X/S}$, 
$\Pic_{X'/S}$ are represented by group schemes $\bPic_{X/S}$, 
$\bPic_{X'/S}$ in view of Proposition \ref{prop:rep}.

\begin{proposition}\label{prop:vect}
With the above notation and assumptions, the connected component
of the zero section, $\bPic_{X/S}^0$, exists and is isomorphic to $G$.
If $A$ is locally projective, then so is $X$.
\end{proposition}

\begin{proof}
Since $Y = S$, the natural map $(\Pic_{X/S},Y) \to \Pic_{X/S}$
is an isomorphism (as follows e.g. from the exact sequence
(\ref{eqn:exa})).
Thus, $\Pic_{X/S} \cong (\Pic_{X'/S},Y')$ by Lemma \ref{lem:rig}.
Also, we have an exact sequence of group schemes
\[ 0 \longrightarrow \bG_{m,S} \longrightarrow V_{Y'}^* 
\longrightarrow  V \longrightarrow 0 \]
with the notation of Subsection \ref{subsec:pic},  
and hence an exact sequence of \'etale sheaves
\[ 0 \longrightarrow V \longrightarrow (\Pic_{X'/S},Y') 
\longrightarrow \Pic_{X'/S} \longrightarrow 0 \]
by Corollary \ref{cor:long} (this also follows directly from 
the exact sequence (\ref{eqn:exa})).

For each $s \in S$, we have 
$X'_s = \wA_s \times_{k(s)} \bP(\cQ^{\vee}_s \oplus k(s))$
and hence $\bPic^0_{X'_s/k(s)}\cong A_s$. In particular, 
$\bPic^0_{X'_s/k(s)}$ is smooth of dimension independent of $s$.
By \cite[Prop.~9.5.20]{Kleiman}, it follows that $\bPic^0_{X'/S}$
exists and its fiber at any $s \in S$ is $\bPic^0_{X'_s/k(s)}$. 
Thus, the projection $\pi : X' \to \wA$ yields an isomorphism
\[ \pi^*: A = \bPic^0_{\wA/S} 
\stackrel{\cong}{\longrightarrow} \bPic^0_{X'/S}. \]

Moreover, $\pi$ sits in a commutative diagram of rigidifiers 
in the (generalized) sense of \cite[Def.~2.1.1]{Raynaud}
\[ \CD
Y' @>{\id}>> Y' \\
@V{\iota'}VV @V{I_S(\gamma)}VV \\
X' @>{\pi}>> \wA \\
\endCD \]
which yields a commutative diagram of exact sequences
\[ \CD
0 @>>> V @>>> (\Pic_{\wA/S},Y') @>>>\Pic_{\wA/S} @>>> 0 \\
& & @V{\id}VV @V{\pi^*}VV @V{\pi^*}VV \\
0 @>>> V @>>> (\Pic_{X'/S},Y') @>>>\Pic_{X'/S} @>>> 0 \\
\endCD \]
in view of \cite[\S 1]{Onsiper}. It follows that 
$(\bPic^0_{\wA/S},Y')$ and $(\bPic^0_{X'/S},Y')$ exist and are isomorphic
via the commutative diagram of exact sequences
\[ \CD
0 @>>> V @>>> (\bPic^0_{\wA/S},Y') @>>> \bPic^0_{\wA/S} @>>> 0 \\
& & @V{\id}VV @V{\pi^*}VV @V{\pi^*}VV \\
0 @>>> V @>>> (\bPic^0_{X'/S},Y') @>>> \bPic^0_{X'/S} @>>> 0. \\
\endCD \]
On the other hand, the commutative diagram of rigidifiers
\[ \CD
Y' @>{I_S(\gamma)}>> I_S(\omega_{\wA}) \\
@VVV @VVV \\
\wA @>{\id}>> \wA \\
\endCD \]
yields a commutative diagram of exact sequences
\[ \CD
0 @>>> V(\omega_{\wA}) @>>> (\bPic^0_{\wA/S},I_S(\omega_{\wA})) 
@>>>\bPic^0_{\wA/S} @>>> 0 \\
& & @V{\gamma}VV @VVV @V{\id}VV \\
0 @>>> V @>>> (\bPic^0_{\wA/S},Y') @>>>\bPic^0_{\wA/S} @>>> 0. \\
\endCD \]
Moreover, the top line in the above diagram is the universal
vector extension of $A$, in view of 
\cite[Chap.~I, (2.6)]{Mazur-Messing}. It follows that the
bottom line is the extension (\ref{eqn:vect}). Finally, the local
projectivity assertion follows from the construction and 
Proposition \ref{prop:rep}.
\end{proof}

\subsection{Extensions by quasi-split tori}
\label{subsec:semi}

Consider a torus $T$ over $S$. We say that $T$ is quasi-split
if there exists a finite \'etale Galois cover $f: S' \to S$
with group $\Gamma$, and a permutation $\bZ[\Gamma]$-module $P$ 
satisfying
\[ T_{S'} \cong \bG_{m,S'} \otimes_{\bZ} P \]
as group schemes over $S'$ equipped with an action of $\Gamma$,
compatible with its action on $S'$. (Recall that a $\bZ[\Gamma]$-module 
is said to be a permutation module if it admits a $\Gamma$-stable 
$\bZ$-basis).

When $S = \Spec(k)$ for a field $k$, the quasi-split tori are
exactly the unit group schemes of finite \'etale $k$-algebras 
(see e.g. \cite[Chap. 2, \S 6.1, Prop.~1]{Voskresenskii}). We shall
extend this to an arbitrary base scheme $S$. 
Let $T$ be a quasi-split torus as above. 
We may decompose the permutation module $P$ as
\[ P = \bigoplus_{i=1}^m \bZ [\Gamma/\Gamma_i], \]
where $\Gamma_1,\ldots,\Gamma_m$ are subgroups of $\Gamma$. 
Consider the scheme $Z$ over $S = S'/\Gamma$ defined by 
\[ Z = \bigsqcup_{i=1}^m S'/\Gamma_i. \] 
Alternatively, we have $Z = (\bigsqcup_{i=1}^n S')/\Gamma$, where 
$n$ denotes the rank of the free $\bZ$-module $P$, and $\Gamma$ acts 
on $\bigsqcup_{i=1}^n S'$ by permuting the $n$ copies of $S'$ with orbits
$\Gamma/\Gamma_1,\ldots,\Gamma/\Gamma_m$.
 
\begin{lemma}\label{lem:qs}
With the above notation, the natural map
$q: \bigsqcup_{i=1}^n S' \to Z$ is a finite \'etale Galois cover
with group $\Gamma$. Also, $Z$ is finite \'etale over $S$, 
and $T \cong V_Z^*$ as $S$-group schemes. 
Conversely, if $Z'$ is a finite \'etale scheme and $S$ is 
connected, then $V_{Z'}^*$ is a quasi-split torus.
\end{lemma}

\begin{proof}
Since $S'/\Gamma_i \cong (S' \times \Gamma/\Gamma_i)/\Gamma$, 
where $\Gamma$ acts diagonally on $S' \times \Gamma/\Gamma_i$, 
we have
\[ Z \cong (S' \times \bigsqcup_{i=1}^m \Gamma/\Gamma_i)/\Gamma, \]
where $\Gamma$ acts diagonally on the right-hand side. In view of
\cite[Exp.~V, Prop.~1.9]{SGA7}, it follows that $q$ is a $\Gamma$-torsor.

To complete the proof of the first assertion, it suffices by descent
to check that the base change $(V_Z)_{S'}$ is finite \'etale over $S'$, 
and $T_{S'} \cong (V_Z^*)_{S'}$
as $S'$-group schemes equipped with a compatible action of $\Gamma$.
Since $V_Z^* = \R_{Z/S}(\bG_{m,Z})$ and Weil restriction commutes with 
base change, we have
$(V_Z^*)_{S'} \cong \R_{Z_{S'}/S'}(\bG_{m,Z_{S'}})$. Moreover, 
\[ Z_{S'} = \bigsqcup_{i=1}^m S' \times_S (S'/\Gamma_i)
\cong \bigsqcup_{i=1}^m (S' \times_S S')/\Gamma_i
\cong \bigsqcup_{i=1}^m (S' \times \Gamma)/\Gamma_i
\cong S' \times \bigsqcup_{i=1}^m \Gamma/\Gamma_i
= \bigsqcup_{i=1}^n S', \]
where the first isomorphism follows from 
\cite[Exp.~V, Prop.~1.9]{SGA7} again, and the second one comes 
from the isomorphism
\[ S' \times \Gamma \stackrel{\cong}{\longrightarrow} S' \times_S S',
\quad (x,g) \longmapsto (gx, x); \]
the composed isomorphism is equivariant for the natural action of
$\Gamma$ on $Z_{S'}$ and its action on $\bigsqcup_{i=1}^n S'$ by permuting
the copies of $S'$. This yields the desired assertions in view of
Lemma \ref{lem:units} (iii).

For the second assertion, we may assume that $Z'$ is connected,
since the product of any two quasi-split tori is easily seen to
be quasi-split. Then, by the classification of finite \'etale
covers in terms of the \'etale fundamental group, there exist
a finite \'etale Galois cover $Z'' \to S$ with group $\Gamma$,
and a subgroup $\Gamma_1 \subset \Gamma$ such that 
$Z' \cong Z''/\Gamma_1$ and the structure map $Z' \to S$ is identified
with the natural morphism $Z''/\Gamma_1 \to Z''/\Gamma = S$.
Thus, $Z' \cong (Z'' \times \Gamma/\Gamma_1)/\Gamma$, where 
$\Gamma$ acts diagonally on $Z'' \times \Gamma/\Gamma_1$. 
Let $S' := Z'' \times \Gamma/\Gamma_1$; then the structure map
$S' \to S$ is a finite \'etale Galois cover with group $\Gamma$.
Moreover, by arguing as in the first part of the proof, we obtain
$\Gamma$-equivariant isomorphisms
\[ Z'_{S'} \cong S' \times_S S'/\Gamma_1 
\cong S' \times \Gamma/\Gamma_1. \]
It follows that 
$V^*_{Z'_{S'}} \cong \bG_{m,S'} \otimes_{\bZ} \bZ[\Gamma/\Gamma_1]$
as an $S'$-torus equipped with a compatible action of $\Gamma$.
Since $V^*_{Z'_{S'}} \cong (V^*_{Z'})_{S'}$, this completes the proof.
\end{proof}

\begin{remark}\label{rem:split}
In the definition of a quasi-split torus $T$, we may replace $S'$ 
with any larger Galois cover. Keeping this in mind, the permutation 
module $P$ is uniquely determined by $T$; the split tori correspond 
of course to the trivial permutation modules. Thus, the direct image 
of $\cO_Z$ under the structure map $Z \to S$ is uniquely determined 
by $T$ as well (this is in fact the Lie algebra of $T$). 
But the $\cO_S$-algebra structure of $\cO_Z$ is not uniquely 
determined by $T$; in fact, the orbits 
$\Gamma/\Gamma_1,\ldots,\Gamma/\Gamma_m$ are not unique, 
since the $\Gamma$-module $\bZ[\Gamma/\Gamma_1]$ does not determine 
the subgroup $\Gamma_1 \subset \Gamma$ up to conjugacy 
(see \cite{Scott}). 
\end{remark}

Next, let $A$ be an abelian scheme and consider the group
$\Ext^1(A,T)$ classifying the extensions of $S$-group schemes
\begin{equation}\label{eqn:semi}
0 \longrightarrow T \longrightarrow G \longrightarrow A 
\longrightarrow 0. 
\end{equation}

\begin{lemma}\label{lem:bw}
With the above notation, there is a canonical isomorphism
\begin{equation}\label{eqn:bw} 
\Ext^1(A,T) \stackrel{\cong}{\longrightarrow} \wA(Z).
\end{equation}
\end{lemma}

\begin{proof}
By \cite[Exp.~VIII, Prop.~3.7]{SGA7}, we have 
a canonical isomorphism (given by push-out)
\[ \Ext^1(A,T) \stackrel{\cong}{\longrightarrow} \Hom(\wT,\wA), \]
where $\wT$ denotes the Cartier dual of $T$. 
Moreover, the pull-back map
\[ \Hom(\wT,\wA) \longrightarrow \Hom^{\Gamma}(\wT_{S'},\wA_{S'}) \]
is an isomorphism by descent theory (see 
\cite[Exp.~VIII, Cor.~7.6]{SGA1}, which applies since every
$\Gamma$-orbit in $\wT_{S'}$ and in $\wA_{S'}$ is contained in
an open affine subscheme). Also, $\wT_{S'}$ is isomorphic to
the constant group scheme $\Hom(P,\bZ)_{S'}$, equivariantly for
the action of $\Gamma$, and hence
\[ \Hom^{\Gamma}(\wT_{S'},\wA_{S'}) \cong 
(P \otimes_{\bZ} \wA(S'))^{\Gamma} \cong 
\bigoplus_{i=1}^m \wA(S')^{\Gamma_i} \cong
\wA(\bigsqcup_{i=1}^m S'/\Gamma_i) \cong \wA(Z). \]
\end{proof}

\begin{remark}\label{rem:russell}
In view of the isomorphism $T \cong \R_{Z/S}(\bG_{m,Z})$ and 
the Weil-Barsotti formula (see \cite[Thm.~18.1]{Oort},
the isomorphism (\ref{eqn:bw}) may be rewritten as 
\[ \Ext^1(A, \R_{Z/S}(\bG_{m,Z})) \cong \Ext^1(A_Z,\bG_{m,Z}). \]
Such an isomorphism has also been obtained by Russell 
(via a very different argument) when $S = \Spec(k)$ for a field 
$k$, and $Z$ is finite but not necessarily \'etale; see 
\cite[Prop.~1.19]{Russell}. In fact, Russell's argument extends
to our relative setting, and yields an isomorphism
$\Ext^1(A,V_Z^*) \cong \wA(Z)$ 
for any finite flat $S$-scheme $Z$.
\end{remark}

We now define
\[ Y':= Z \sqcup S. \]
Then $Y'$ is finite and \'etale over $S$. Moreover, any
extension (\ref{eqn:semi}) yields a morphism $Z \to \wA$ and hence 
a map $Y' \to \wA$, where $S$ is sent to $\wA$ via the zero section 
$s_0$. We also have a closed immersion 
$Z \to \Spec \, \Sym_{\cO_S}(\cA)$,
and hence a closed immersion $Y' \to \bP(\cA \oplus \cO_S)$, where 
$S$ is sent to the section at infinity. This yields a closed immersion
\[ \iota': Y' \longrightarrow \wA \times_S \bP(\cA \oplus \cO_S). \]
Denoting by $\psi: Y'  \to S := Y$ the structure map, we 
may again form the pinching diagram (\ref{eqn:pin}), 
where $\Pic_{X/S}$, $\Pic_{X'/S}$ are represented by group schemes 
$\bPic_{X/S}$, $\bPic_{X'/S}$. We now obtain the same statement
as Proposition \ref{prop:vect}:

\begin{proposition}\label{prop:semi}
With the above notation and assumptions, the connected component of
the zero section, $\bPic_{X/S}^0$, exists and is isomorphic to $G$.
If $A$ is locally projective, then so is $X$.
\end{proposition}

\begin{proof}
As in the proof of Proposition \ref{prop:vect}, the natural map 
$(\Pic_{X/S},Y) \to \Pic_{X/S}$ is an isomorphism, and 
$\Pic_{X/S} \cong (\Pic_{X'/S},Y')$. 

Consider first the case where $T \cong \bG_{m,S}^n$ is split. Then 
with the notation of Subsection \ref{subsec:pic}, the map 
$V_Y^* \to V_{Y'}^*$ may be identified with the diagonal, 
$\delta : \bG_{m,S} \to \bG_{m,S}^{n+1}$. The latter sits 
in an exact sequence of group schemes
\[ \CD 
0 @>>> \bG_{m,S} @>{\delta}>> \bG_{m,S}^{n+1} @>{\gamma}>> 
\bG_{m,S}^n = T @>>> 0,
\endCD \]
where $\gamma(x_1,\ldots,x_n,x_0) := (x_1 x_0^{-1}, \ldots, x_n x_0^{-1})$.
In view of Corollary \ref{cor:long}, this yields an exact sequence
\[ 0 \longrightarrow T \longrightarrow 
(\Pic_{X'/S},Y') \longrightarrow \Pic_{X'/S}
\longrightarrow 0. \]
Next, arguing again as in the proof of Proposition \ref{prop:vect},
we obtain that the projection $\pi: X' \to A$ yields an isomorphism
$\pi^*: A = \bPic^0_{\wA/S} \stackrel{\cong}{\longrightarrow} \bPic^0_{X'/S}$
which extends to an isomorphism of exact sequences
\[ \CD
0 @>>> T @>>> (\bPic^0_{\wA/S},Y') @>>> \bPic^0_{\wA/S} @>>> 0 \\
& & @V{\id}VV @V{\pi^*}VV @V{\pi^*}VV \\
0 @>>> T @>>> (\bPic^0_{X'/S},Y') @>>> \bPic^0_{X'/S} @>>> 0. \\
\endCD \]
Moreover, the top extension
$0 \to T \to (\bPic^0_{\wA/S},Y') \to A \to 0$
is sent to $(s_1, \ldots, s_n)$ by the isomorphism
(\ref{eqn:bw}), as follows from \cite[Prop.~1]{Onsiper} in the case 
where $n = 1$, and from (the proof of) \cite[Cor.~1.1]{Onsiper} 
in the general case. This yields isomorphisms 
$G \cong (\bPic^0_{\wA/S},Y') \cong (\bPic^0_{X'/S},Y') \cong \bPic^0_{X/S}$.

For an arbitrary quasi-split torus $T$, we reduce similarly to
showing that the above extension corresponds to the map 
$Z \to \wA$ under the isomorphism (\ref{eqn:bw}). But this holds 
after the Galois base change $f : S' \to S$ by the preceding step.
Moreover, the pull-back map $\Ext^1(A,T) \to \Ext^1(A_{S'},T_{S'})$
is injective, since it is identified under the isomorphism 
(\ref{eqn:bw}) to the map $\wA(Z) \to \wA(\bigsqcup_{i=1}^n S')$ induced
by the natural morphism $q: \bigsqcup_{i=1}^n S' \to Z$; 
moreover, $q$ is finite and \'etale by Lemma \ref{lem:qs},
and hence is faithfully flat.
\end{proof}

\begin{remark}\label{rem:proj}
In Proposition \ref{prop:vect} (resp. Proposition \ref{prop:semi}), 
we may replace $\bP(\cQ^{\vee} \oplus \cO_S)$ 
(resp. $\bP(\cA \oplus \cO_S)$) 
with any projective space bundle over $S$ that contains $Y'$.
Here, by a projective space bundle, we mean the projectivization
of a locally free sheaf of finite rank over $S$.
\end{remark}

\medskip

\noindent
{\sc Proof of Theorem \ref{thm:first}}.
Note that the quotients $G/T$, $G/V$ exist and sit in extensions
\[ 0 \longrightarrow V \longrightarrow G/T \longrightarrow A
\longrightarrow 0, \quad 
0 \longrightarrow T \longrightarrow G/V \longrightarrow A
\longrightarrow 0. \]
The sum of these extensions is the extension (\ref{eqn:first}),
since the natural map $G \to G/T \times_A G/V$ is easily
seen to be an isomorphism. Moreover, these extensions
yield morphisms of schemes $Y'_1 := I_S(\cQ^{\vee}) \to \wA$, where 
$V = V(\cQ^{\vee})$, and $Y'_2 := Z \sqcup S \to \wA$; in turn,
this yields closed immersions 
$Y'_1 \hookrightarrow \wA \times_S \bP(\cQ^{\vee} \oplus \cO_S)$
and $Y'_2 \hookrightarrow \wA \times_S \bP(\cA \oplus \cO_S)$.
Now consider the composition of the closed immersions
\[ Y' := Y'_1 \sqcup Y'_2 \longrightarrow 
\wA \times_S (\bP(\cQ^{\vee} \oplus \cO_S) \sqcup \bP(\cA \oplus \cO_S))
\longrightarrow 
\wA \times_S \bP(\cQ^{\vee} \oplus \cO_S \oplus \cA \oplus \cO_S) =: X', \]
and the natural map $Y' = Y'_1 \sqcup Y'_2 \to S \sqcup S =: Y$. 
Then the statement follows by combining Lemma \ref{lem:fib}, 
Propositions \ref{prop:vect} and \ref{prop:semi}, and 
Remark \ref{rem:proj}.

\section{Relative unit groups}
\label{sec:rel}

\subsection{Definition and first properties} 
\label{subsec:def}

Throughout this section, we fix a base field $k$ and choose 
an algebraic closure $\bar{k}$. We denote by $k^{\sep}$
the separable closure of $k$ in $\bar{k}$, and by $\Gamma$ the
Galois group of $k^{\sep}/k$.

We shall consider (commutative) artinian $k$-algebras. Given such 
an algebra $A$, we denote by $\mu^A$ its group scheme of units, 
introduced in Remark \ref{rem:field}. Then 
$\mu^A = \R_{A/k}(\bG_{m,A})$, where $\R_{A/k}$ denotes the Weil 
restriction (see e.g. \cite[App.~A.5]{CGP}). Thus, $\mu^A$ is 
a connected affine algebraic group with Lie algebra $A$. Also, 
we may uniquely decompose $A$ as a direct product 
$A_1 \times \cdots \times A_n$ of local $k$-algebras; then 
$\mu^A \cong \mu^{A_1} \times \cdots \times \mu^{A_n}$.

When $A$ is a subalgebra of an algebra $B$, we have 
$\mu^A \subset \mu^B$ (by Remark \ref{rem:field} again) and we set
\[ \mu^{B/A} := \mu^B/\mu^A. \]
Then $\mu^{B/A}$ is a connected affine algebraic group, that we shall
call the \emph{relative unit group}; its Lie algebra is $B/A$.
Any chain of algebras $A \subset B \subset C$ yields an exact sequence
of algebraic groups
\begin{equation}\label{eqn:chain}
0 \longrightarrow \mu^{B/A} \longrightarrow \mu^{C/A}
\longrightarrow \mu^{C/B} \longrightarrow 0.
\end{equation}
Also, note that $\mu^{(A \times A)/A} \cong \mu^A$ in view of the exact
sequence
\[ \CD
0 @>>> \mu^A @>>> \mu^{A \times A} = \mu^A \times \mu^A
@>{f}>> \mu^A @>>> 0,
\endCD \]
where $f(x,y) = x y ^{-1}$. 

Our main motivation for studying relative unit groups comes from 
the following:

\begin{proposition}\label{prop:pic}
When $k$ is perfect, the algebraic groups of the form 
$\mu^{B/A}$ are exactly the affine parts of Picard varieties of 
projective varieties with finite non-normal locus.
\end{proposition}

\begin{proof}
Let $X$ be such a variety, and denote by $\varphi : X' \to X$ 
the normalization. Then $X'$ is projective, and 
we have an exact sequence 
$0 \to \mu^{B/A} \to \Pic^0(X) \to \Pic^0(X') \to 0$
for appropriate algebras $A \subset B$ (see Remark \ref{rem:field}).
Moreover, $\Pic^0(X')$ is an abelian variety by 
\cite[Thm.~9.5.4, Rem.~9.5.6]{Kleiman}. Thus, $\mu^{B/A}$ 
is the affine part of $\Pic^0(X)$.

Conversely, given algebras $A \subset B$, we may embed 
$\Spec(B)$ in some projective space $\bP$, and form the
pinching diagram
\[ \CD
\Spec(B) @>>> \bP \\
@VVV @VVV \\
\Spec(A) @>>> X.\\
\endCD \]
Then $\mu^{B/A}= \Pic^0(X)$ in view of Remark \ref{rem:field} 
again.
\end{proof}

Since relative unit groups are interesting in their own right, 
we shall consider them in more detail than is needed for 
applications to Picard varieties. We begin with the following:

\begin{examples}\label{ex:def}
(i) Let $K/k$ be a finite separable field extension. We may assume 
that $K \subset k^{\sep}$; we then denote by $\Gamma_K \subset \Gamma$ 
the Galois group of $k^{\sep}/K$. Then $\mu^K$ is a torus with character 
module $\bZ[\Gamma/\Gamma_K]$.
It follows that $\mu^{K/k}$ is a torus as well, with character module 
the kernel of the augmentation map $\bZ[\Gamma/\Gamma_K] \to \bZ$. 

\smallskip

\noindent
(ii) More generally, consider an algebra $A$ which is separable
(or equivalently, \'etale). Then $\mu^A$ is a quasi-split torus; 
moreover, all quasi-split tori are obtained in this way, as recalled
in Subsection \ref{subsec:semi}.

\smallskip

\noindent
(iii) Let $A := k \oplus I$, where $I$ is an ideal of square $0$.
Then $\mu^{A/k}$ is the vector group associated with $I$.

\smallskip

\noindent
(iv) Assume that $\charc(k) = p > 0$ and $[k^{1/p} : k] = p$. 
Let $K := k^{1/p}$ and choose $t \in k \setminus k^p$. Then
$\mu^{K/k}$ is isomorphic to the closed subgroup scheme of
$\bG_a^p$ defined by 
$x_0^p + t x_1^p + \cdots + t^{p-1} x_{p-1}^p = x_{p-1}$
(see \cite[Prop.~VI.5.3]{Oesterle}). In particular, $\mu^{K/k}$ 
is unipotent, and contains no copy of $\bG_a$ in view of 
\cite[Lem.~VI.5.1]{Oesterle}. In other words, $\mu^{K/k}$ is 
$k$-wound in the sense of Tits (see \cite[V.3]{Oesterle} and 
also \cite[Def.~B.2.1, Cor.~B.2.6]{CGP}). 
\end{examples}

Next, we collect basic properties of relative unit groups, in 
a series of lemmas.

\begin{lemma}\label{lem:alg}
{\rm (i)} Let $I$ be an ideal of an algebra $A$. Then the quotient map 
$A \to A/I$ yields an epimorphism $\gamma: \mu^A \to \mu^{A/I}$. 
If $I$ is nilpotent, then $\Kern(\gamma) = 1 + I$ with an obvious
notation.

\smallskip

\noindent
{\rm (ii)} Let $I \subset A \subset B$, where $I$ is a nilpotent 
ideal of $B$. Then the natural map 
$\mu^{B/A} \to \mu^{(B/I)/(A/I)}$ is an isomorphism.

\smallskip

\noindent
{\rm (iii)} Let $A,A'$ be subalgebras of an algebra $B$. 
Then the natural map $\iota : \mu^{A'/(A \cap A')} \to \mu^{B/A}$
is a closed immersion.

\smallskip

\noindent
{\rm (iv)} Let $K/k$ be a finite extension of fields. Then 
the base change $\mu^{B/A}_K$ is isomorphic to 
$\mu^{B \otimes_k K/A \otimes_k K}$ as a $K$-group scheme.
\end{lemma}

\begin{proof}
(i) To show that $\gamma$ is an epimorphism, it suffices to check
that the induced map $\mu^A(\bar{k}) \to \mu^{A/I}(\bar{k})$ is 
surjective, since $\mu^A$ and $\mu^{A/I}$ are algebraic groups.
Thus, we may assume that $k$ is algebraically closed; also, we may
reduce to the case that $A$ is local. Then its maximal ideal
$\fm$ is nilpotent, and $A^* \cong k^* \times (1 + \fm)$ while
$(A/I)^* \cong k^* \times (1 + \fm/I)$. So the map 
$A^* \to (A/I)^*$ is surjective as desired. The assertion on 
$\Kern(\gamma)$ is obvious.

(ii) By (i), we have a commutative diagram of exact sequences
\[ \CD
0 @>>> 1 + I @>>> \mu^A @>{\gamma_A}>> \mu^{A/I} @>>> 0 \\
& & @V{\id}VV @VVV @VVV \\
0 @>>> 1 + I @>>> \mu^B @>{\gamma_B}>> \mu^{B/I} @>>> 0 \\
\endCD \]
which yields the assertion.

(iii) Clearly, $\iota$ induces an injective morphism on Lie 
algebras. Arguing as in the proof of (i), it suffices to show that
$\iota$ is also injective on $\bar{k}$-points. But this follows
from the equality $(A \cap A')^* = A^* \cap A'^*$.

(iv) Since exact sequences of group schemes are preserved by field
extensions, it suffices to show that 
$\mu^A_K = \mu^{A \otimes_k K}$,
where the right-hand side is understood as a $K$-group scheme. 
Let $R$ be a $K$-algebra; then
$\mu^A_K(R) = \mu^A(R) = (A \otimes_k R)^* = 
(A \otimes_k K \otimes_K R)^* = \mu^{A \otimes_k K}(R)$.
\end{proof}

\begin{lemma}\label{lem:red}
Let $A \subset B$ be algebras, $I$ (resp. $J$) the nilradical
of $A$ (resp. $B$), and set $A_\red := A/I$, $B_\red := B/J$. 

\smallskip

\noindent
{\rm (i)} $A_\red \subset B_\red$ and we have an exact sequence 
of algebraic groups
\begin{equation}\label{eqn:red}
0 \longrightarrow (1 + J)/(1 + I) \longrightarrow \mu^{B/A}
\longrightarrow \mu^{B_\red/A_\red} \longrightarrow 0.
\end{equation}

\smallskip

\noindent
{\rm (ii)} Let $A_\sep \subset A_\red$ be the largest separable
subalgebra, and define likewise $B_\sep$. Then 
$A_\sep = A_\red \cap B_\sep$ and the homomorphism
$\iota : \mu^{B_\sep/A_\sep} \to \mu^{B_\red/A_\red}$ is a closed 
immersion. Moreover, the exact sequence (\ref{eqn:red})
splits canonically over $\mu^{B_\sep/A_\sep}$.

\smallskip

\noindent
{\rm (iii)} $(1 + J)/(1 + I)$ has a composition series with 
subquotients $\bG_a$.
\end{lemma}

\begin{proof}
(i) Since $A \cap J = I$, the map $A_\red \to B_\red$ is injective.
Moreover, by Lemma \ref{lem:alg} (i), we have a commutative diagram 
of exact sequences
\[ \CD
0 @>>> 1 + I @>>> \mu^A @>>> \mu^{A_\red} @>>> 0 \\
& & @VVV @VVV @VVV \\
0 @>>> 1 + J @>>> \mu^B @>>> \mu^{B_\red} @>>> 0. \\
\endCD \]
This yields the exact sequence (\ref{eqn:red}).

(ii) Clearly, $A_\sep \subset A_\red \cap B_\sep$; also, the 
opposite inclusion holds since every subalgebra of a separable
algebra is separable. This yields the desired equality,
and in turn the assertion on $\iota$ in view of Lemma 
\ref{lem:alg} (iii).

Denote by $B' \subset B$ 
the preimage of $B_\sep$ and define $A'\subset A$ similarly;
then $A' = A \cap B'$. By a special case of the Wedderburn-Malcev 
theorem (see e.g. \cite[Thm.~(72.19)]{Curtis-Reiner}),
the exact sequence of algebras $0 \to J \to B' \to B_\sep \to 0$ 
has a unique splitting. Thus, 
$B' = B_\sep \oplus J \supset A_\sep \oplus I = A'$. This yields
compatible splittings in the exact sequences
\[ \CD
0 @>>> 1 + I @>>> \mu^{A'} @>>> \mu^{A_\sep} @>>> 0 \\
& & @VVV @VVV @VVV \\
0 @>>> 1 + J @>>> \mu^{B'} @>>> \mu^{B_\sep} @>>> 0, \\
\endCD \]
and hence the desired splitting.

(iii) We may replace $A$ (resp. $B$) with its subalgebra 
$k \oplus I$ (resp. $k \oplus J$), and hence assume that 
$A,B$ are local with residue field $k$. Then the subspaces
$B_m := k \oplus (I + J^m)$, where $m \geq 1$, form a decreasing
sequence of subalgebras of $B$, with $B_1 = B$ and $B_m = A$ for 
$m \gg 0$. Using the exact sequence (\ref{eqn:chain}) and the
inclusion $(I + J^m)^2 \subset I + J^{m+1}$, we may thus assume
that $J^2 \subset I$. Then $I$ is an ideal of $J$, and hence we
may further assume that $I = 0$ by using Lemma \ref{lem:alg} (ii). 
In that case, $(1 + J)/(1 + I) = 1 + J$ is a vector group, since
$J^2 = 0$.
\end{proof}

\begin{lemma}\label{lem:ser}
Let $A \subset B$ be reduced algebras and write
$A = \prod_{i=1}^m K_i$, $B = \prod_{j=1}^n L_j$, where $K_i,L_j$ are 
fields. Then $\mu^{B/A}$ has a composition series with subquotients
$\mu^{L_j/K_i}$ (where $K_i \hookrightarrow L_j$) and possibly 
$\mu^{K_i}$. Moreover, all the $\mu^{L_j/K_i}$ occur with 
multiplicity $1$.
\end{lemma}

\begin{proof}
Let $e_1,\ldots, e_m$ be the primitive idempotents of $A$. Then 
\[ A = \prod_{i=1}^m K_i = \prod_{i=1}^m A e_i \subset
\prod_{i=1}^m B e_i = B, \]
and each $B e_i$ is a subalgebra of $B$. Thus,
$\mu^{B/A} = \prod_{i=1}^m \mu^{Be_i/Ae_i}$, and hence we may assume
that $A$ is a field, say $K$. Then 
$K \subset K^n \subset \prod_{j=1}^n L_j = B$, so that 
(\ref{eqn:chain}) yields an exact sequence
\[ 0 \longrightarrow \mu^{K^n/K} \longrightarrow \mu^{B/A}
\longrightarrow \prod_{j=1}^n \mu^{L_j/K} \longrightarrow 0. \]
We may factor the diagonal inclusion $K \subset K^n$ as
$K \subset K^2 \subset \cdots \subset K^n$, where each
$K^i$ is embedded in $K^{i+1}$ via 
$(x_1,\ldots,x_i) \mapsto (x_1,\ldots, x_i,x_i)$. Thus,
$\mu^{K^n/K}$ has a composition series with subquotients 
$\mu^{K^{i+1}/K^i}$. Moreover, the map
\[ \mu^{K^{i+1}} = (\mu^K)^{i+1} \longrightarrow \mu^K, \quad
(x_1,\ldots, x_{i+1}) \longmapsto x_i x_{i+1}^{-1} \]
is an epimorphism with kernel $\mu^{K^i}$, and hence yields 
an isomorphism $\mu^{K^{i+1}/K^i} \cong \mu^K$.
\end{proof}

\subsection{Tori}
\label{subsec:tor}

We keep the notation of Subsection \ref{subsec:def}.
We first record the following observation, probably well-known 
but that we could not locate in the literature: 

\begin{lemma}\label{lem:nil}
Let $K/k$ be a finite extension of fields and denote by $K_\sep$ 
the separable closure of $k$ in $K$. Then $K_\sep \otimes_k \bar{k}$ 
is the largest reduced subalgebra of $K \otimes_k \bar{k}$. 

In particular, the nilradical of $K \otimes_k \bar{k}$ has 
dimension $[K:k] - [K_\sep:k]$ as a $\bar{k}$-vector space; moreover,
$\mu^{K_\sep}$ is the maximal torus of $\mu^K$.
\end{lemma}

\begin{proof}
We have an isomorphism of $\bar{k}$-algebras
$K_\sep \otimes_k \bar{k} \cong \prod_{i=1}^m \bar{k}$, where 
$m := [K_\sep:k]$. Also, we may assume that $k$ has characteristic 
$p > 0$ (since there is nothing to prove in characteristic $0$). 
Then $x^{p^n} \in K_\sep$ for $n \gg 0$ and all $x \in K$. Thus, 
$x^{p^n} \in K_\sep \otimes_k \bar{k}$ for $n \gg 0$ and all
$x \in K \otimes_k \bar{k}$. It follows that
$K \otimes_k \bar{k} = (K_\sep \otimes_k \bar{k}) \oplus I$,
where $x^{p^n} = 0$ for $n \gg 0$ and all $x \in I$. 
This yields the assertions on $K_\sep \otimes_k \bar{k}$ 
and on the nilradical of $K \otimes_k \bar{k}$. 
As a consequence, $\mu^{K_\sep \otimes_k \bar{k}}$ is the maximal 
torus of $\mu^{K \otimes_k \bar{k}}$; the assertion on $\mu^{K_\sep}$ 
follows in view of Lemma \ref{lem:alg} (iv).
\end{proof}

We may now describe the maximal tori of relative unit groups:

\begin {proposition}\label{prop:max}
Let $A \subset B$ be algebras, $I \subset J$ their nilradicals,
$A_\red := A/I \subset B/J =: B_\red$ the associated quotients, and 
$A_\sep \subset B_\sep$ the largest separable subalgebras of these
quotients. 

\smallskip

\noindent
{\rm (i)} $\mu^{B_\sep/A_\sep}$ is the maximal torus of $\mu^{B/A}$.

\smallskip

\noindent
{\rm (ii)} If $B_\red = B_\sep$ (and hence $A_\red = A_\sep$; this 
holds e.g. if $k$ is perfect), then 
\[ \mu^{B/A} \cong (1 + J)/(1 + I) \times \mu^{B_\red/A_\red}, \]
where $(1 + J)/ (1 + I)$ is unipotent and $\mu^{B_\red/A_\red}$ 
is a torus.
\end{proposition}

\begin{proof}
(i) Given an exact sequence of connected algebraic groups
$0 \to G_1 \to G \to G_2 \to 0$, the sequence of maximal
tori $0 \to T(G_1) \to T(G) \to T(G_2) \to 0$ is exact as well.
Thus, it suffices to show that $\mu^{B_\sep}$ is the maximal torus
of $\mu^B$. For this, we may assume that $B$ is reduced, in view 
of Lemma \ref{lem:alg} (i). Then $B$ is a direct product of fields, 
and we conclude by Lemma \ref{lem:nil}.

(ii) follows from (i) in view of Lemma \ref{lem:red}. 
\end{proof}

\begin{remark}\label{rem:stab}
With the notation of the above proposition, the maximal torus
$T$ of $\mu^{B/A}$ sits in an exact sequence
$0 \to \mu^{A_\sep} \to \mu^{B_\sep} \to T \to 0$. 
Also, $\mu^{A_\sep}$, $\mu^{B_\sep}$ are quasi-split tori, as seen in 
Example \ref{ex:def} (i). By
\cite[Chap.~2, \S 4.7, Thm.~2]{Voskresenskii}, it follows that $T$ 
is stably rational (this is a restrictive condition on tori, e.g., 
if the group $\bZ/2 \bZ \times \bZ/2\bZ$ is a quotient of 
$\Gamma$, then some tori of dimension $3$ are not stably rational; 
see \cite[Chap.~2, \S 4.10]{Voskresenskii}). We do not know whether 
all stably rational tori can be obtained as relative unit groups.
\end{remark}

\begin{remark}\label{rem:real}
Every connected affine algebraic group $G$ over the field $\bR$ 
of real numbers is the Picard variety of some projective variety.
Indeed, $G \cong V \times T$, where $V \cong \bG_{a,\bR}^m$ is 
a vector group, and $T$ a torus; moreover, by 
\cite[Chap.~4, \S 10.1]{Voskresenskii}, $T$ is isomorphic to a direct
product of copies of $\mu^{\bR}$, $\mu^{\bC}$, and $\mu^{\bC}/\mu^{\bR}$. 
Using Theorem \ref{thm:first} and Lemma \ref{lem:fib}, we reduce 
to the cases where $G = \mu^{\bC}$ or $G = \mu^{\bC}/\mu^{\bR}$. 
In the latter case, 
we may choose a smooth projective rational curve $X'$ containing 
a closed point $Y'$ with residue field $\bC$; pinching via the 
structure map $Y' \to Y := \Spec(\bR)$ yields the desired variety, 
as can be checked by arguing as in the proof of Proposition 
\ref{prop:pic}. In the former case, we replace $Y'$ with $Z'$, 
where $Z'$ is the disjoint union of $Y'$ and a closed point with 
residue field $\bR$, and pinch via the structure map again.

The same result holds for any real closed field $k$, with the same 
proof. Yet we do not know whether it extends to all connected (not 
necessarily affine) algebraic groups over $k$. The example in 
\cite[p.~505]{Onsiper} suggests a negative answer to that question.
\end{remark}

Next, we characterize those relative unit groups that are tori:

\begin{proposition}\label{prop:car}
With the notation of Proposition \ref{prop:max}, 
the following are equivalent:

\smallskip

{\rm (i)} $\mu^{B/A}$ is a torus.

\smallskip

{\rm (ii)} $I = J$ and $B_\red$ is separable over $k$ (hence
so is $A_\red$).
\end{proposition}

\begin{proof}
(i)$\Rightarrow$(ii) We must have $I = J$ by Lemma \ref{lem:red}.
In view of Lemma \ref{lem:alg} (ii), we may thus assume that $B$ 
(and hence $A$) is reduced. Write $A = \prod K_i$ and $B = \prod L_j$
as in Lemma \ref{lem:ser}. By that lemma, $\mu^{L/K}$ must be a torus
whenever $K = K_i \hookrightarrow L_j = L$. Thus, the base change
$\mu^{L/K}_{\bar{k}}$ is a torus over $\bar{k}$. This is equivalent
to $\mu^{L \otimes_k \bar{k}/K \otimes_k \bar{k}}$ being a torus, in view
of Lemma \ref{lem:alg} (iv). Using the exact sequence
\[ 0 \longrightarrow \mu^{K \otimes_k \bar{k}} \longrightarrow 
\mu^{L \otimes_k \bar{k}} \longrightarrow 
\mu^{L \otimes_k \bar{k}/K \otimes_k \bar{k}} \longrightarrow 0 \]
and Lemma \ref{lem:red}, it follows that
$K \otimes_k \bar{k}$ and $L \otimes_k \bar{k}$ have the same
nilradical. By Lemma \ref{lem:nil}, this yields
\[ [K:k] - [K_\sep:k] = [L:k] - [L_\sep:k]. \]
Since $K_\sep = K \cap L_\sep$, we have 
\[ \dim_k(K + L_\sep) = [K:k] + [L_\sep:k] - [K_\sep:k] = [L:k], \]
and hence $K + L_\sep = L$; in particular, $L = K L_\sep$. 
Since the extension $L_\sep/K_\sep$ is separable and $K/K_\sep$
is purely inseparable, $L_\sep$ and $K$ are linearly disjoint 
over $K_\sep$ (as follows e.g. from Mac Lane's criterion). 
As a consequence, 
\[ [L:K_\sep] = [L_\sep:K_\sep] [K:K_\sep]. \] 
On the other hand, 
$ [L:K_\sep] = \dim_{K_\sep}(K + L_\sep) = 
[K:K_\sep] + [L_\sep:K_\sep] -1$.
Thus, we obtain 
\[ ([L_\sep:K_\sep] -1) ([K:K_\sep] -1) = 0, \]
and hence $L_\sep = K_\sep$ or $K = K_\sep$. In the former case, 
we have $L = K + L_\sep = K$. In the latter case, $L = L_\sep$, 
i.e., $L$ is separable over $k$.

(ii)$\Rightarrow$(i) By Lemma \ref{lem:red}, we have 
$\mu^{B/A} \cong \mu^{B_\red/A_\red}$. Moreover, $\mu^{B_\red/A_\red}$
is a torus in view of Proposition \ref{prop:max}. 
\end{proof}

\subsection{Unipotent groups}
\label{subsec:uni}

Throughout this subsection, we consider algebras $A \subset B$ with
nilradicals $I \subset J$ and associated quotients 
$A_\red = A/I \subset B/J = B_\red$. We first obtain an (easy)
characterization of those relative unit groups that are unipotent:

\begin{proposition}\label{prop:uni}
{\rm (i)} When $\charc(k) = 0$, $\mu^{B/A}$ is unipotent if and
only if $A_\red = B_\red$.

\smallskip

\noindent
{\rm (ii)} When $\charc(k) = p > 0$, $\mu^{B/A}$ is unipotent
if and only if $b^{p^n} \in A$ for $n \gg 0$ and all $b \in B$. 
Equivalently, the extension $L/K$ is purely inseparable for any 
inclusion $K \subset L$, where $K$ (resp. $L$) is a residue field 
of $A$ (resp. $B$).
\end{proposition}

\begin{proof}
(i) follows from Lemma \ref{lem:red} (ii), since $\mu^{B_\red/A_\red}$
is a torus by Proposition \ref{prop:max}.

(ii) Recall that $\mu^{B/A}$ is unipotent if and only if its group 
of $\bar{k}$-points is $p^n$-torsion for $n \gg 0$. Since
$\mu^{B/A}(\bar{k}) = (B \otimes_k \bar{k})^*/(A \otimes_k \bar{k})^*$,
this is in turn equivalent to the condition that 
$b^{p^n} \in (A \otimes_k \bar{k})^*$ for $n \gg 0$
and all $b \in (B \otimes_k \bar{k})^*$. As the $\bar{k}$-vector 
space $B \otimes_k \bar{k}$ is spanned by $(B \otimes_k \bar{k})^*$, 
this is also equivalent to $b^{p^n} \in A \otimes_k \bar{k}$ for 
$n \gg 0$ and all $b \in B \otimes_k \bar{k}$, and hence to
$b^{p^n} \in A$ for $n \gg 0$ and all $b \in B$.

The equivalence with the condition on residue fields follows readily
from the structure of $A$ and $B$.
\end{proof}

\begin{remark}\label{rem:geom}
The above results may be reformulated in terms of the morphism 
\[ \psi : Z := \Spec(B) \longrightarrow \Spec(A) =: Y \] 
associated with the inclusion of algebras $A \subset B$ 
(so that $Y$, $Z$ are finite, and $\psi$ is surjective). 
For example, Proposition \ref{prop:uni} means that $\mu^{B/A}$ 
is unipotent if and only if $\psi$ is a universal homeomorphism.

Likewise, when $A$ contains no ideal of $B$, Proposition 
\ref{prop:car} means that $\mu^{B/A}$ is a torus if and only if 
$Y$ and $Z$ are \'etale.

Also, Lemma \ref{lem:ser} may be reformulated and slightly 
sharpened as follows: if $Z$ (and hence $Y$) is reduced, then  
$\mu^{B/A}$ has a composition series with subquotients 
$\mu^{k(z)/k(y)}$, where $y \in Y$ and $z \in \psi^{-1}(y)$,
and possibly $\mu^{k(y)}$. Moreover, all the $\mu^{k(z)/k(y)}$
occur with multiplicity $1$, and $\mu^{k(y)}$ with multiplicity
$\vert \psi^{-1}(y) \vert -1$.
\end{remark}

Next, we show that certain unipotent relative unit groups 
are $k$-wound, generalizing Example \ref{ex:def} (iv). For this, 
we shall need:

\begin{lemma}\label{lem:wound}
Let $k \subset K \subset L$ be a tower of finite extensions of fields, 
where $K/k$ is separable. Then every homomorphism of algebraic groups 
$h: \bG_a \to \mu^{L/K}$ is constant.
\end{lemma}

\begin{proof}
Since $\mu^K$ is a torus, every extension
$0 \to \mu^K \to G \to \bG_a \to 0$ splits by 
\cite[Exp.~XVII, Thm.~6.1.1]{SGA3}. In view of the exact sequence
$0 \to \mu^K \to \mu^L \to \mu^{L/K} \to 0$, it follows that any
homomorphism $h : \bG_a \to \mu^{L/K}$ lifts to a homomorphism 
$\tilde{h}: \bG_a \to \mu^L$. We may view $\tilde{h}$ as a $k[t]$-point 
of $\mu^L$, i.e., $\tilde{h} \in L[t]^* = L^*$. Since 
$\tilde{h}(0) = 1$, it follows that $\tilde{h}$ is constant.
\end{proof}

With the assumptions of the above lemma, if in addition 
$L/K$ is purely inseparable, then it follows that the unipotent 
group $\mu^{L/K}$ is $k$-wound (this also results from 
\cite[Prop.~V.7, Lem.~VI.5.1]{Oesterle}). 
We do not know whether $\mu^{L/K}$ is $k$-wound when
$K/k$ is no longer assumed to be separable.

\medskip

Returning to the setting of algebras $A \subset B$ with nilradicals
$I \subset J$, we now obtain a succession of elementary results 
which will readily imply Theorem \ref{thm:second}:

\begin{lemma}\label{lem:small}
{\rm (i)} The maximal ideals of $J$ are exactly the hyperplanes
containing $J^2$.

\smallskip

\noindent
{\rm (ii)} There exists a flag of subspaces
$I = I_0 \subset I_1 \subset \cdots \subset I_n = J$
such that $I_i$ is a maximal ideal of $I_{i+1}$ for all $i$. 
In particular, $n = \dim(J) - \dim(I)$. 

\smallskip

\noindent
{\rm (iii)} $J^{2^n} \subset I$.
\end{lemma}

\begin{proof}
(i) Let $K$ be a maximal ideal of $J$. Then $J/K$ is a nilpotent
algebra having no proper ideal. Hence $\dim(J/K) =1$ and
$(J/K)^2 = 0$. In other words, $K$ is a hyperplane of $J$ containing
$J^2$. Conversely, any such hyperplane is clearly a maximal ideal.

(ii) Let $m$ be the largest integer such that $J^m = 0$. Then we 
have a flag of subspaces 
$I \subset I + J^m \subset I + J^{m-1} \subset \cdots \subset
I + J^2 \subset J$.
Choose a complete flag of subspaces $I_i$ refining this partial
flag. Then each $I_i$ can be written as $I + V$ for some subspace
$V$ such that $J^{j+1} \subset V \subset J^j$ for some $j$. Since
$(I+V)(I + J^j) = I^2 + IV + I J^j + V J^j \subset I + J^{j+1}$,
we see that each $I + V$ is an ideal of $I + J^j$. This implies
the assertion.

(iii) By (i), we have $I_i^2 \subset I_{i+1}$ for all $i$. 
This yields the statement by induction.
\end{proof}

Next, assume that $k$ has characteristic $p > 0$. Let 
$U := (1 + J)/(1 + I)$ and $n := \dim(U) = \dim(J) - \dim(I)$. 
Then $U$ is an iterated extension of $n$ copies of $\bG_a$ 
by Lemma \ref{lem:red} (iii); hence the commutative group 
$U(\bar{k})$ is $p^n$-torsion. Let $m$ be the smallest positive 
integer such that $U(\bar{k})$ is $p^m$-torsion; then $m \leq n$.
We say that $U$ has period $p^m$. 

We shall use repeatedly the following observation:

\begin{lemma}\label{lem:max}
With the above notation, assume that $U$ has maximal period
$p^n$. Let $I'$ be a subalgebra of $J$ containing $I$. Then
the connected unipotent groups $(1 + I')/(1 + I)$ and 
$(1 + J)/(1 + I')$ have maximal period as well.
\end{lemma}

\begin{proof}
This follows readily from the exact sequence
(a special case of (\ref{eqn:chain}))
\[ 0 \longrightarrow (1 + I')/(1+I) \longrightarrow U 
\longrightarrow (1 + J)/(1 + I') \longrightarrow 0. \]
\end{proof}

We now consider successively the cases where $p \geq 5$, $p  = 3$
and $p = 2$ (the latter turns out to be much less straightforward):

\begin{lemma}\label{lem:period}
With the above notation, we have $m < n$ when $p \geq 5$ and 
$n \geq 2$. 
\end{lemma}

\begin{proof}
We argue by contradiction, and assume that $U$ has maximal period.
Since $n \geq 2$, there exists a subalgebra $I_2 \subset J$ such
that $I \subset I_2$ and $\dim(I_2) = \dim(I) + 2$ 
(by Lemma \ref{lem:small} (ii)). By Lemma \ref{lem:max},
the $2$-dimensional subgroup $(1 + I_2)/(1 + I)$ is not $p$-torsion. 
On the other hand, $I_2^4 \subset I$ by Lemma \ref{lem:small} (iii). 
If $p \geq 5$, then $(1 + x)^p = 1 + x^p \in I$ for all $x \in I_2$, 
a contradiction.
\end{proof}

\begin{lemma}\label{lem:period3}
With the above notation, we have $m < n$ when $p = 3$ and $n \geq 3$. 
\end{lemma}

\begin{proof}
We adapt the argument of Lemma \ref{lem:period}. By Lemma \ref{lem:small}, 
we may choose a subalgebra $I_3 \subset J$ such that $I \subset I_3$ 
and $\dim(I_3) = \dim(I) + 3$. By Lemma \ref{lem:max} again,
the $3$-dimensional subgroup $(1 + I_3)/(1 + I)$ is not $9$-torsion,
if $U$ has maximal period. But $I_3^8 \subset I$ by Lemma \ref{lem:small} 
again; this yields a contradiction.
\end{proof}

\begin{lemma}\label{lem:period2}
With the above notation, we have $m < n$ when $p = 2$ and $n \geq 3$.
\end{lemma}

\begin{proof}
We argue again by contradiction, and assume that $U$ has maximal
period. We may reduce to the case where $n = 3$ as in the proof of 
Lemma \ref{lem:period3}. To analyze $(1 + J)/(1 + I)$, we begin with 
some further reductions.

If $I$ contains an ideal $J'$ of $J$, then the natural homomorphism
\[ (1 + J)/(1 + I) \longrightarrow (1 + J/J')/(1 + I/J') \]
is an isomorphism by Lemma \ref{lem:alg} (ii). Thus, we may assume 
that $I$ contains no nonzero ideal of $J$.

Also, if there exists a subalgebra $I'$ of $J$ such that $I + I' = J$,
then the natural homomorphism 
\[ (1 + I')/(1 + I \cap I) \longrightarrow (1 + J)/(1 + I) \]
is an isomorphism, as follows from Lemma \ref{lem:alg} (iii) in 
view of the equality $\dim(I'/I\cap I') = \dim(J/I)$. Thus, we may
assume that there exists no proper subalgebra $I'$ of $J$ such that
$I + I' = J$. By Lemma \ref{lem:small} (ii), this is equivalent to
the assumption that $I \subset I'$ for any maximal ideal $I'$ of $J$.
In view of Lemma \ref{lem:small} (i), we may thus assume that
$I \subset J^2$.

By Lemma \ref{lem:max}, it follows that the group $(1 + J)/(1 + J^2)$ 
has maximal period. But $(1 + J)/(1 + J^2) \cong 1 + J/J^2$ is a vector
group, and hence has period $2$. Hence $\dim(J/J^2) = 1$. 
By Nakayama's lemma, we then have
\[ J = t k[t]/(t^{m+1}) = \langle x, x^2, \ldots, x^m \rangle \] 
for some $x \in J$ and a unique integer $m \geq 1$. Then 
\[ J^2 = (x^2) = \langle x^2, x^3, \ldots, x^m \rangle \] 
is the unique maximal ideal of $J$. Moreover, our reductions
mean that $I \subset \langle x^2, x^3, \ldots, x^m \rangle$ and 
$x^m \notin I$.

Consider $I' := \langle I,x^m \rangle \subset J$; this is a subalgebra 
of codimension $2$ of $J$, which contains $I$ as a maximal ideal. 
By Lemma \ref{lem:small} (ii), $I'$ is a maximal ideal of $J^2$;
hence $I' \supset J^4$ by that lemma, (iii). Since 
$J^4 = \langle x^4, x^5, \ldots, x^m \rangle$, there exist $a, b \in k$
such that $(a,b) \neq (0,0)$ and 
\[ I' = \langle a x^2 + b x^3, x^4, x^5, \ldots, x^m \rangle. \]
Moreover, $a \neq 0$: otherwise, 
$I'= \langle x^3, x^4, x^5, \ldots, x^m \rangle$ is an ideal of $J$, 
so that $(1 + J)/(1 + I')$ has dimension $2$ and period $2$; 
this yields a contradiction in view of Lemma \ref{lem:max}.

By Lemma \ref{lem:small} (iii) again, we have $I'^2 \subset I$. Thus, 
$I$ contains $x^8, x^9, \ldots$ and also $(ax^2 + bx^3)x^5$; in 
particular, $x^7 \in I$. Likewise, $(a x^2 + b x^3) x^4 \in I$
so that $x^6 \in I$. By our reductions, it follows that $x^6 = 0$. 
Also, $a^2 x^4 + b^2 x^6 = (a x^2 + b x^3)^2 \in I$; thus, $x^4 \in I$. 
Since $x$ generates the nilpotent algebra $J$, this yields
$y^4 \in I$ for all $y \in J$. As a consequence, $(1 + J)/(1 + I)$
has period at most $4$, a contradiction.
\end{proof}

\medskip

\noindent
{\sc Proof of Theorem \ref{thm:second}}.
We argue again by contradiction, and assume that $W_n$ is isogenous 
to $\bPic_{X/k}$ for some projective variety $X$ with finite non-normal 
locus. In particular, $U := \Pic^0(X)$ is unipotent. 
By \cite[Chap.~VII, no.~10, Prop.~9]{Serre}, $U$ has maximal period 
$p^n$, where $n := \dim(U)$. 
On the other hand, there exist algebras $A \subset B$ such that 
$U \cong \mu^{B/A}$, by Proposition \ref{prop:pic}.  
Since $k$ is perfect and $U$ is unipotent, we must have 
$U \cong (1 + J)/(1 + I)$ by Lemma \ref{lem:red}. But then
Lemmas \ref{lem:period}, \ref{lem:period3} and \ref{lem:period2} 
yield a contradiction.

\begin{remark}\label{rem:nil}
Consider the algebra $B := k[x]/(x^4)$, and its subalgebra $A$
generated by $x^2 + x^3$ (of square $0$). Then $U := \mu^{B/A}$ 
is a connected unipotent group of dimension $2$. If $p = 3$
(resp. $p = 2$), then $U$ has period $9$ (resp. $4$) since 
$x^3, x^2 \notin A$. By \cite[Chap.~VII, no.~10, Prop.~9]{Serre}
again, it follows that $U$ is isogenous to $W_2$. In particular,
the statement of Theorem \ref{thm:second} is optimal.
\end{remark}

\medskip

\noindent
{\bf Acknowledgements}. I warmly thank St\'ephane Druel, Philippe 
Gille, St\'ephane Guillermou, Emmanuel Peyre and Henrik Russell 
for helpful discussions or e-mail exchanges. Special thanks are 
due to Sylvain Brochard for his careful reading of a preliminary 
version of this article, and his detailed comments and corrections. 
Needless to say, the remaining mistakes are mine.

\end{document}